\documentclass[12pt,reqno]{amsart}
\advance \topmargin by -\headheight
\advance \topmargin by -\headsep
\evensidemargin \oddsidemargin
\usepackage{amsmath}
\usepackage{amsfonts}
\usepackage{stmaryrd}
\usepackage{amssymb}
\usepackage{amsthm}
\usepackage{csquotes}

\usepackage{mathrsfs}
\usepackage{dsfont}
\usepackage{bm}
\usepackage{cite}
\usepackage{soul}

\numberwithin{equation}{section}
\renewcommand\vec{\bm}
\newcommand{\n}[1]{\|{#1}\|}
\newcommand{\GL}{\text{GL}}
\newcommand{\PGL}{\text{PGL}}

\newtheorem{theorem}{Theorem}[section]

\newtheorem{lemma}[theorem]{Lemma}
\newtheorem{Proposition}[theorem]{Proposition}
\newtheorem{Conjecture}[theorem]{Conjecture}
\newtheorem{Corollary}[theorem]{Corollary}

\usepackage{mathtools}
\DeclarePairedDelimiter{\ceil}{\lceil}{\rceil}
\DeclarePairedDelimiter{\floor}{\lfloor}{\rfloor}

\title[A quadratic Vinogradov mean value theorem in finite fields]{A quadratic Vinogradov mean value theorem in finite fields}

\author{Sam Mansfield}
\address{Department of Mathematics, University of Bristol, Bristol, BS8 1UG, UK}
\email{sam.mansfield@bristol.ac.uk}

\author{Akshat Mudgal}
\address{Mathematical Institute, University of Oxford, Oxford OX2 6GG, UK}
\email{mudgal@maths.ox.ac.uk}

\date{} 

\renewcommand\vec{\bm}

\begin{document}

\begin{abstract}
Let $p$ be a prime, let $s \geq 3$ be a natural number and let $A \subseteq \mathbb{F}_p$ be a non-empty set satisfying $|A| \ll p^{1/2}$. Denoting $J_s(A)$ to be the number of solutions to the system of equations
\[ \sum_{i=1}^{s} (x_i - x_{i+s}) =  \sum_{i=1}^{s} (x_i^2 - x_{i+s}^2) = 0, \]
with $x_1, \dots, x_{2s} \in A$, our main result implies that
\[  J_s(A) \ll |A|^{2s - 2 - 1/9}. \]
This can be seen as a finite field analogue of the quadratic Vinogradov mean value theorem. Our techniques involve a variety of combinatorial geometric estimates, including studying incidences between Cartesian products $A\times A$ and a special family of modular hyperbolae.

\end{abstract}
\maketitle

\section{Introduction}

Our aim in this paper is to obtain relatives of the well-known Vinogradov's mean value theorem over finite fields. In particular, given $s,k \in \mathbb{N}$ and some finite set $A \subseteq \mathbb{Z}$, an important problem in analytic number theory has been to count the number of solutions $J_{s,k}(A)$ to the system of equations
\begin{equation} \label{vino4}
    \sum_{i=1}^{s} (x_i^j - x_{i+s}^j) = 0 \ \ \ (1 \leq j \leq k), \end{equation}
where $x_1, \dots, x_{2s} \in A$. This has been analysed extensively in the setting when $A = [N]$, where we denote $[N] = \{1, 2, \dots, N\}$ for every $N \in \mathbb{N}$, in part due to its connections to Waring's problem and estimates on the zero-free region for the Riemann zeta function, see \cite{Wo2014, Wo2019} and the references therein. Here,  for every $s,k,N \in \mathbb{N}$ and $\epsilon >0$, one has 
\begin{equation} \label{vincon}  N^s + N^{2s - k(k+1)/2} \ll_{s,k} J_{s,k}([N]) \ll_{s,k, \epsilon} N^{\epsilon} (N^s + N^{2s - k(k+1)/2}) ,
\end{equation} 
with the lower bound arising from elementary considerations and the upper bound being the so-called main conjecture in Vinogradov’s mean value theorem. While the $k=1$ case of the above upper bound is trivial and the $k = 2$ case follows in a straightforward manner after applying estimates for the divisor function, even the $k =3$ case of this was wide open until very recent work of Wooley \cite{Wo2015}. The $k \geq 4$ case of this was then resolved in the famous work of Bourgain--Demeter--Guth \cite{BDG2016}, with an alternate number theoretic proof provided later by Wooley \cite{Wo2019}. Both these works further showed that for any set $A \subseteq [N]$, one has 
\begin{equation} \label{dsre}
    J_{s,k}(A) \ll_{s,k,\epsilon}  N^{\epsilon} (|A|^s + |A|^{2s - k(k+1)/2} )
\end{equation} 
for every $s,k \in \mathbb{N}$ and $\epsilon>0$, see also work of Guth--Maldague--Wang\cite{GMW2020} and Guo--Li--Yung \cite{GLY2021} for stronger bounds in the $k=2$ case. On the other hand, in this more general setting,  it has been conjectured that the $N^{\epsilon}$ factor above can be improved to $|A|^{\epsilon}$.

\begin{Conjecture} \label{bdcon}
 Let $A \subseteq \mathbb{Z}$ be a finite set, let $s,k$ be natural numbers and let $\epsilon >0$ be a real number. Then
 \[ J_{s,k}(A) \ll_{s,k, \epsilon} |A|^{\epsilon} ( |A|^s + |A|^{2s - k(k+1)/2} ) . \]
\end{Conjecture}

This was originally stated for the $k=2$ case by Bourgain--Demeter \cite{BD2015} with the more general case recorded in \cite{Mu2020, Wo2023}. Moreover, while the $k=1$ case of this is trivial, unlike \eqref{vincon}, even the $k=2$ case of this remains open. Note that in the setting when the set $A$ satisfies $|A| \gg N^{\delta}$ for some fixed $\delta >0$, Conjecture \ref{bdcon} is implied by \eqref{dsre}, with the implicit constant in the Vinogradov notation depending on $\delta$. 
On the other hand, when $A$ is much more sparse, say, $|A| \ll \log \log N$, then \eqref{dsre} perform worse than the trivial bound $J_{s,2}(A) \ll |A|^{2s-2}.$ In such regimes, the best known \emph{diameter-free} estimate for $J_{s,2}(A)$, that is, an upper bound that does not depend on $N$, arises from work of the second author in \cite{Mu2023a}. In particular, the latter implies that
\begin{equation} \label{gel}
J_{s,2}(A) \ll_{s} |A|^{2s - 3 + \delta_s}
\end{equation} 
for every $s \geq 3$ and $A \subseteq [N]$, where $\delta_3 = 1/2$ and $\delta_s = (1 - c)\cdot 2^{-s+ 2}$ for every $s \geq 4$, with $c = 1/1812$. Thus, one can obtain exponents arbitrarily close to conjectured exponent in Conjecture \ref{bdcon} when $k=2$ and $s$ is large.

In this paper, we consider the $k=2$ case of Conjecture \ref{bdcon} in the finite field setting, and so, given a prime number $p$ and some set $A \subseteq \mathbb{F}_p$, we denote $J_{s}(A)$ to count the number of solutions to \eqref{vino4} when $k=2$ and $x_1, \dots, x_{2s} \in A$. Here, for any $s \geq 3$ and for any $A \subseteq \mathbb{F}_p$ of the form $A = [N] (\mod p)$ with $N < p$, we have the elementary lower bound
\begin{equation} \label{lbdl}
 J_{s}(A) \gg_s |A|^{2s-1}/p + |A|^{2s-3} ,
 \end{equation}
and so, one can not obtain power savings analogous to Conjecture \ref{bdcon}  when $p^{1/2} \ll |A| \leq p$.  In fact, this regime can be analysed using standard methods from analytic number theory and additive combinatorics, and indeed, in \S2, we record the following estimate.

\begin{Proposition} \label{lset}
    Let $s \geq 3$ be an integer and let $A \subseteq \mathbb{F}_p$ be a non-empty set. Then 
    \[ J_s(A) \ll_s |A|^{2s-1}/p +  |A|^2p^{s-2} (\log |A|)^{2s} \]
\end{Proposition}

This matches the lower bound in \eqref{lbdl}, up to multiplicative constants, when we have $|A| \gg p^{ 1/2 +  1/(4s-6)} (\log p)^2$ and $s \geq 3$. Thus, we focus on the sparse set case, that is, when $|A| \ll p^{1/2}$; the latter often being inaccessible to the aforementioned methods. Our main result in this paper dispenses non-trivial bounds in this regime for arbitrary sets $A \subseteq \mathbb{F}_p$.

\begin{theorem} \label{3c}
    Let $s \geq 3$ and let $A \subseteq \mathbb{F}_p$ satisfy $|A| \ll p^{1/2}$. Then we have
    \[ J_{s}(A) \ll |A|^{2s - 2 - 1/9} .\]
\end{theorem}

The trivial bound here is $J_s(A) \ll |A|^{2s - 2}$ and we improve upon this by a factor of $|A|^{1/9}$. As for lower bounds, noting $\eqref{lbdl}$, one sees that there are large subsets $A \subseteq \mathbb{F}_p$ satisfying $|A| \gg_s p^{1/2}$ such that $J_s(A) \gg_s |A|^{2s - 3}$ for every $s \geq 3$. It would be desirable to show that 
\[ J_{s}(A) \ll_{s,\epsilon}  |A|^{2s-3 + \epsilon} \]
for every $\epsilon >0$ and for all sufficiently small sets $A \subseteq \mathbb{F}_p$, and in fact, such an estimate would deliver the $k=2$ case of Conjecture \ref{bdcon} in a straightforward manner. It would be interesting to even obtain bounds similar in strength to \eqref{gel} in the finite field setting for sparse sets.

We remark that for larger values of $s$, we can prove slightly better exponents than those presented in Theorem \ref{3c}, at the cost of restricting to sparser sets $A \subseteq \mathbb{F}_p$, and this is precisely the content of our next result.

\begin{theorem} \label{4c}
Let $s \geq 4$ and let $A \subseteq \mathbb{F}_p$ satisfy $|A| \ll_{s} p^{\frac{15}{13(s-1)}}$. Then we have that
    \[ J_s(A) \ll_s |A|^{2s - 2 - 1/7 + \eta_s}, \]
    where $\eta_s = (4/11)^{s-3} \cdot (2/63)$.
\end{theorem}

We prove these results by reducing the problem of counting $J_s(A)$ to estimating incidences between various point sets and translates of curves in $\mathbb{F}_p^2$. For instance, one of the key ideas in the proof of Theorem \ref{3c} considers counting incidences between a special family of modular hyperbolae and point sets of the form $A \times A$ in $\mathbb{F}_p^2$, see Lemma \ref{hypinc}. Similarly, we use point-line incidences in $\mathbb{F}_p^2$ to obtain slightly better exponents when $s \geq 4$ in Theorem \ref{4c}. This accounts for another reason as to why there is a gap in the exponents available in \eqref{gel} and Theorems \ref{3c} and \ref{4c}, since incidence results over finite fields are few and often non-optimal, while, for instance, sharp point-line incidence results have been available in the Euclidean setting since the classical work of Szemer\'{e}di--Trotter \cite{ST1983}.

We further mention that our techniques can deliver results of a sum-product flavour. This would not be a surprise to experts in the area since our methods rely crucially on incidence geometric ideas, the latter being a standard set of tools used for studying the sum-product problem since the key work of Elekes \cite{El1997}. In order to elucidate upon this further, we recall the sum-product conjecture raised by Erd\H{o}s-Szemer\'{e}di \cite{ES1983}, which states that given any $s \in \mathbb{N}$, any $\epsilon >0$ and any finite set $A \subseteq \mathbb{Z}$, one has
\[ |sA| + |A^{(s)}| \gg_{s, \epsilon} |A|^{s-\epsilon},  \]
where $sA = \{a_1 + \dots + a_s : a_1, \dots, a_s \in A\}$ and $A^{(s)} = \{ a_1\dots a_s : a_1, \dots, a_s \in A\}.$
While a significant body of work has been done towards this problem, this conjecture remains wide open, with the best known bounds towards this recorded in the work of Rudnev--Stevens \cite{RS2022} for the $s=2$ case, and in a paper by P\'{a}lv\"{o}lgyi--Zhelezov \cite{PZ2020} for the case when $s$ is large. In particular, building upon the breakthrough work of Bourgain--Chang \cite{BC2004}, the latter paper shows that for any $s \in \mathbb{N}$ and for any finite $A \subseteq \mathbb{Z}$, one has
\[ |sA| + |A^{(s)}| \gg_{s} |A|^{c \log s/ \log \log s},  \]
for some absolute constant $c>0$. A natural ``energy" analogue of this problem would be to study, for any $s \in \mathbb{N}$ and any finite $A \subseteq \mathbb{N}$, the number of solutions $T_{s}(A)$ to the system of equations
\[ a_1 \dots a_s = a_{s+1} \dots a_{2s} \ \  \text{and} \ \  a_1 + \dots + a_s = a_{s+1} + \dots + a_{2s}, \]
with $a_1, \dots, a_s \in A$. Here, by noting the sum-product philosophy, one may expect that
\[  T_s(A) \ll_{s, \epsilon} |A|^{s+ \epsilon}  \]
for any $s \in \mathbb{N}$ and $\epsilon >0$. Moreover, upon applying the Cauchy-Schwarz inequality, one sees that the above conjectured estimate would imply major progress towards the sum-product conjecture by dispensing the estimate
\[ |sA| |A^{(s)}| \gg_{s,\epsilon} |A|^{s- \epsilon}, \]
for every $s \in \mathbb{N}$, every $\epsilon >0$ and every finite $A \subseteq \mathbb{N}$. The best known upper bound for $T_{s}(A)$, when $A\subseteq \mathbb{N}$, can be deduced from recent work of the second author \cite{Mu2023b} which implies that
\[ T_{s}(A) \ll_s |A|^{2s - c \log s/ \log \log s} , \]
for some absolute constant $c>0$. 

Returning to the finite field setting, given a non-empty set $A \subseteq \mathbb{F}_p$, we denote $T(A)$ to be the number of solutions to the system
\begin{equation} \label{fv1}
    a_1 + a_2 + a_3 = a_4 + a_5 + a_6 \ \ \text{and} \ \ a_1 a_2 a_3 = a_4 a_5 a_6 
    \end{equation} 
with $a_1, \dots, a_6 \in A$. With this in hand, we state another application of our methods.

\begin{theorem} \label{smp}
    Let $A \subseteq \mathbb{F}_p$ be a finite set with $|A| \ll p^{1/3}$. Then
    \[ T(A) \ll |A|^{4 - 1/11} .\]
\end{theorem}

It is worth remarking that we prove the above result more to show that non-trivial bounds may be found for $T_s(A)$ in finite fields as opposed to the motivation of obtaining sharper sum-product estimates in finite fields. Indeed, the bounds that one may obtain for $|3A| + |A^{(3)}|$ by applying the Cauchy-Schwarz inequality in conjunction with Theorem \ref{smp} are significantly worse than the state of the art sum-product estimates in finite fields recorded in the work of Mohammadi--Stevens \cite{MS2023}. In particular, the main result in \cite{MS2023} implies that for any $A \subseteq \mathbb{F}_p$ with $|A| \ll p^{1/2}$ and for any $\epsilon >0$, one has
\[ |2A| + |A^{(2)}| \gg_{\epsilon} |A|^{5/4 - \epsilon}. \]

We briefly mention some applications of Theorem \ref{3c}. For instance, Theorem \ref{3c} delivers a weaker version of a result of Shkredov--Shparlinski \cite[Lemma 2.10]{SS2018} in a straightforward manner.

\begin{Corollary} \label{cor1}
Let $A \subseteq \mathbb{F}_p$ satisfy $|A| \ll p^{1/2}$ and $|2A| \leq K|A|$ and let $S = \{a^2 : a \in A\}$. Then
\[ |3S|  \gg  |A|^{1 + 1/9} K^{-6} .\]
\end{Corollary}

This can be interpreted as a finite field analogue of results on sumsets of convex subsets of real numbers. For instance, we refer the reader to work of Elekes--Nathanson--Ruzsa \cite{ENR2000}, who showed that for any convex function $f: \mathbb{R} \to \mathbb{R}$ and for any finite set $A \subseteq \mathbb{R}$ satisfying $|2A| \leq K|A|$, one has 
\[ |2f(A)| \geq |A|^{3/2} K^{-1},\]
see also \cite{SW2022} and the references therein for the best known estimates in this direction.

It is a  well known phenomenon in analytic number theory and harmonic analysis that one can employ estimates on energies akin to $J_s(A)$ to furnish bounds on moments of weighted exponential sums, see, for instance, \cite{BD2015, CG2007, GGP2022, Mu2022}. Thus, we define, for every $n \in \mathbb{N}$ and every $q \geq 1 $ and every function $g : \mathbb{F}_p^n \to \mathbb{R}$, the $L^q$ norm $\n{g}_{L^q}$ and $l^q$ norm $\n{g}_q$ of $g$ to be
\begin{equation} \label{hr123}   
\n{g}_{L^q} = \big( p^{-n} \sum_{\vec{x} \in \mathbb{F}_p^n} |g(\vec{x})|^q   \big)^{1/q} \ \ \text{and}  \ \ \n{g}_q =  \big(  \sum_{\vec{x} \in \mathbb{F}_p^n} |g(\vec{x})|^q   \big)^{1/q} 
\end{equation} 
Moreover, given a function $\mathfrak{a} : \mathbb{F}_p \to \mathbb{R}$, we define the exponential sum $F_{\mathfrak{a}} : \mathbb{F}_p^2 \to \mathbb{C}$ as 
\begin{equation} \label{ep3}
F_{\mathfrak{a}}(x,y) = \sum_{n \in \mathbb{F}_p} \mathfrak{a}(n) e( (x n + y n^2)/p), 
\end{equation} 
where $e(\theta) = e^{2 \pi i \theta}$ for every $\theta \in \mathbb{R}$. With this in hand, we state the second corollary of Theorem \ref{3c}.

\begin{Corollary} \label{cor2}
Let $A \subseteq \mathbb{F}_p$ satisfy $|A| \ll p^{1/2}$ and let $\mathfrak{a} : \mathbb{F}_p \to \mathbb{R}$ be a function supported on the set $A$. Then
\[ \n{F_{\mathfrak{a}}}_{L^6} \ll (\log |A| + 1)^{1/2}|A|^{4/27} \n{\mathfrak{a}}_{2}.  \]
\end{Corollary}

We note that the Cauchy-Schwarz bound in this setting is 
\begin{equation} \label{vach}
 \n{F_{\mathfrak{a}}}_{L^6} \ll |A|^{1/6} \n{\mathfrak{a}}_{2}  , 
 \end{equation}
see \S6 for a proof of this. Thus, we improve upon this estimate by a factor of $|A|^{1/54 - o(1)}$. On the other hand, we may set $p$ to be some large prime, $I = \{1,2,\dots, \floor{10^{-1} p^{1/3}}\}$ and $A = I (\mod p)$ and $\mathfrak{a} = \mathds{1}_A$ to be the indicator function of the set $A$ to deduce the lower bound
\[  \n{F_{\mathfrak{a}}}_{L^6}  = J_3(A)^{1/6} = J_{3,2}(I)^{1/6}   \gg  |I|^{1/2} ( \log |I|+ 1)^{1/6} = \n{\mathfrak{a}}_{2}  ( \log|A| + 1)^{1/6}, \]
where the inequality in the middle may be discerned from inequality (2.51) in \cite{Bo1993}. Corollary \ref{cor2} can be interpreted as a finite field analogue of the discrete Fourier restriction phenomenon for the parabola in $\mathbb{Z}^2$, see \cite{Bo1993, GMW2020, GLY2021} for more details about the latter. Moreover, Corollary \ref{cor2} recovers the conclusion of Theorem \ref{3c} up to a factor of $(\log |A| + 1)^3$; this can be seen by noting \eqref{jj3}, setting $\mathfrak{a}= \mathds{1}_A$ in Corollary \ref{cor2} and applying orthogonality.

We recall that Theorem \ref{3c} delivers non-trivial estimates for all sets $A \subseteq \mathbb{F}_p$ satisfying $|A| \ll p^{1/2}$. When $A$ is significantly sparser, say, $|A| \ll \log \log \log p$, then one may employ results of Grosu \cite{Gr2014} to transfer to the setting when one is studying the corresponding problem for arbitrary finite subsets of $\mathbb{C}$. Moreover, since we have a complex analogue of the Szemer\'{e}di--Trotter theorem (see \cite{To2014, Za2015}), one may be able to obtain much stronger versions of Theorem \ref{3c}, akin to the estimate recorded in \eqref{gel}, in the case when $|A| \ll \log \log \log p$, by following the circle of ideas presented in \cite{Mu2023a}. 

We now provide a brief outline of our paper. In \S2, we present various incidence geometric and combinatorial lemmata that we will require throughout our paper, along with the proof of Proposition \ref{lset}. We employ \S3 to reduce the proof of Theorem \ref{3c} to proving various incidence results between special families of hyperbolae and point sets of the form $A\times A$. We record the proofs of the latter estimates in \S4. We employ \S5 to prove Theorems \ref{4c} and \ref{smp}, and in \S6, we present the proofs of Corollaries \ref{cor1} and \ref{cor2}.

\textbf{Notation.} In this paper, we use Vinogradov notation, that is, we write $X \gg_{z} Y$, or equivalently $Y \ll_{z} X$, to mean $X \geq C_{z} |Y|$ where $C$ is some positive constant depending on the parameter $z$. We further write $X = O_{z}(Y)$ to mean $X \ll_{z} Y$. Given a subset $A$ of some abelian group $G$, we use $\mathds{1}_A$ to denote the indicator function of the set $A$, that is, $\mathds{1}_A(n) = 1$ whenever $n \in A$ and $\mathds{1}_A(n) = 0$ for every $n \in G \setminus A$. We will use boldface to denote elements $\vec{v} = (v_1, v_2)$ of $\mathbb{F}_p^2$.

\textbf{Acknowledgements.} The authors are grateful to Misha Rudnev for many valuable discussions. The authors would also like to thank Jori Merikoski, Olly Roche-Newton, Audie Warren, Trevor Wooley and Josh Zahl for helpful comments. The authors would like to thank Igor  Shparlinski for pointing to the reference \cite{SS2018}. The second author is supported by Ben Green's Simons Investigator Grant, ID 376201.


\section{Preliminaries}

Let $G$ be a group with the addition operation, and let $A, B$ be finite, non-empty subsets of $G$. We then denote
\[ A+ B =  \{ a + b : a \in A, b \in B\} \ \ \text{and} \ \ -B = \{ - b : b \in B \}. \]
Moreover given integer $s \geq 2$, we recall the definition of the $s$-fold sumset
\[ sA = \{ a_1 + \dots + a_s : a_1, \dots, a_s \in A\}.\]
Next, let $G$ be a group with the product operation, and let $A,B$ be finite non-empty subsets of $G$. Then we define
\[ A \cdot B  = \{ a \cdot b : a \in A, b \in B\} \ \ \text{and} \ \ B^{-1} = \{ b^{-1} : b \in B\}. \]

Furthermore, as in \eqref{hr123}, given any function $w : \mathbb{F}_p \times \mathbb{F}_p \to [0, \infty)$, we write
\[\n{w}_2^2 = \sum_{\vec{x} \in \mathbb{F}_p \times \mathbb{F}_p} w(\vec{x})^2 \ \ \text{and} \ \ \n{w}_1 = \sum_{\vec{x} \in \mathbb{F}_p \times \mathbb{F}_p} w(\vec{x}) \ \  \text{and} \ \  \n{w}_{\infty} = \max_{\vec{x} \in \mathbb{F}_p \times \mathbb{F}_p} w(\vec{x}).  \]

In our proof of Theorem \ref{3c}, we will be interested in estimating the number of solutions to the equation
\[ (x - a)(y - b) = c, \]
with $x, y \in A$, for fixed choices of $a,b \in \mathbb{F}_p$ and $c\in \mathbb{F}_p \setminus \{0\}$. This is equivalent to counting solutions to the equation
\begin{equation} \label{mob1}
y = b + \frac{c}{x - a} =  \frac{bx - ab + c}{x - a} 
\end{equation}   
with $x,y \in A$. This naturally leads us to consider M\"{o}bius transformations. In particular, given $a,b,c,d \in \mathbb{F}_p$ with $ad - bc \neq 0$, we define the M\"{o}bius transformation map $g : \mathbb{F}_p \to \mathbb{F}_p$ as
\[ g(x) = \frac{ax + b}{cx + d} \]
for every $x \in \mathbb{F}_p$. One can check that the set of M\"{o}bius transformations on $\mathbb{F}_p$, under composition, forms a group isomorphic to $\PGL_2(\mathbb{F}_p)$. Indeed, for each such M\"{o}bius transformation, one may associate the matrix $M \in \GL_2(\mathbb{F}_p)$ given by
\[  M = \begin{pmatrix}
a & b  \\
c & d
\end{pmatrix} , \]
with this association only being unique up to multiplication by a non-zero scalar, since for any $\lambda \neq 0$, the M\"{o}bius transformations associated with the matrices 
\[   \begin{pmatrix}
a & b  \\
c & d
\end{pmatrix}  \ \ \text{and} \ \ \begin{pmatrix}
\lambda a & \lambda b  \\
\lambda c & \lambda d
\end{pmatrix}  \]
are the same. From the preceding discussion, one can deduce that whenever the matrices $M_1$ and $M_2$ are associated with the M\"{o}bius transformations $g_1$ and $g_2$ respectively, then the matrix $M_1 M_2$ is associated with the M\"{o}bius transformation $g_1 \circ g_2$.

Given a set $A \subseteq \mathbb{F}_p$ and a set $H$ of M\"{o}bius transformations of the form \eqref{mob1}, estimates for the number of solutions to the equation $h(a_1) = a_2$, with $h \in H$ and $a_1, a_2 \in A$, have been analysed in work of Bourgain \cite{Bo2012}, Shkredov \cite{Shk2021}, and Rudnev--Wheeler \cite{RW2022}, with such results having various applications to sum-product type problems, see \cite{JM2022, RW2022, Shk2021, Shk2023}. For our purposes, we will require the following incidence estimate from \cite{RW2022}.

\begin{lemma} \label{rich}
   Let $k,r \geq 2$ be integers, let $A$ be a non-empty subset of $\mathbb{F}_p$ with $|A| \ll p^{1/2}$, let $g_1, \dots, g_r : \mathbb{F}_p \to \mathbb{F}_p$ be M\"{o}bius transformations such that 
   \[ \min_{1 \leq i \leq r} \sum_{a_1, a_2 \in A} \mathds{1}_{a_2 = g_i(a_1)} \geq k .\]
     Then we have that
     \[ r \ll |A|^{7} k^{-5}. \]
\end{lemma}

Another incidence estimate that we will require in this paper, and in particular, in the proofs of Theorems \ref{4c} and \ref{smp}, is the following result given by Stevens--de-Zeeuw \cite{SZ2017}.

\begin{lemma} \label{sdz}
Let $P$ be a set of points in $\mathbb{F}_p^2$ and let $L$ be a set of lines in $\mathbb{F}_p^2$ such that 
\[ |L|^{13} \ll p^{15} |P|^2 \ \ \text{and}  \ \ |P|^{7/8} \leq |L| \leq |P|^{8/7}.\]
Then we have that
\[ \sum_{p \in P} \sum_{l \in L} \mathds{1}_{p \in l}  \ll |P|^{11/15} |L|^{11/15}. \]
\end{lemma}

We may combine this with the standard Cauchy-Schwarz bound (see \cite[Lemma 1]{SZ2017})
\[ \sum_{p \in P} \sum_{l \in L} \mathds{1}_{p \in l}  \ll \min\{ |P|^{1/2} |L| + |P|, |L|^{1/2} |P| + |L|\} \]
to get that whenever $|L|^{13} \ll p^{15}$, we have 
\begin{equation} \label{sz2}
 \sum_{p \in P} \sum_{l \in L} \mathds{1}_{p \in l} \ll  |P|^{11/15} |L|^{11/15} + |P| + |L|.  
\end{equation}

We would also require the following lemma \cite[Lemma 3.3]{Mu2022} that allows one to convert incidence estimates between points and lines to their weighted versions. We note that while the version recorded in \cite{Mu2022} is stated only for points and varieties in $\mathbb{R}^d$, its proof extends to the finite field setting mutatis mutandis.

\begin{lemma} \label{wt}
    Let $a,b \in (1/2, 1)$ be real numbers, let $C >0$ be a constant, let $P$ be a non-empty subset of $\mathbb{F}_p \times \mathbb{F}_p$, let $L$ be a set of lines in $\mathbb{F}_p \times \mathbb{F}_p$ such that for every non-empty $L' \subseteq L$ and every non-empty $P' \subseteq P$, we have
    \[ \sum_{x \in P'} \sum_{l \in L'} \mathds{1}_{x \in l} \leq C ( |P'|^a |L'|^b + |P'| + |L'|). \]
    Then for every $w: P \to \mathbb{N} \cup \{0\}$ and for every $w': L \to \mathbb{N} \cup \{0\}$, we have
    \[  \sum_{x \in P} \sum_{l \in L} w(x) w'(l) \mathds{1}_{x \in v} \ll C  ( \n{w}_2^{2-2a} \n{w}_1^{2a - 1} \n{w'}_2^{2- 2b}\n{w'}_1^{2b-1} + \n{w}_1 \n{w'}_{\infty} + \n{w}_{\infty} \n{w'}_{1}). \]
\end{lemma}

We will now present a standard result in additive combinatorics known as the Pl\"{u}nnecke-Ruzsa inequality \cite[Corollary 6.29]{TV2006}, which, in the situation when $|A+B| \leq K|B|$, allows us to efficiently bound many-fold sumsets of $A$.

\begin{lemma} \label{prin}
    Let $A,B$ be finite, non-empty subsets of some abelian group $G$ satisfying $|A+B|\leq K|B|$. Then for every $k \in \mathbb{N}$, we have that
    \[  | kA- kA| \leq K^{2k} |B|. \]
\end{lemma}

We will conclude this section by presenting the proofs of inequalities \eqref{lbdl} and Proposition \ref{lset}. We first consider the former, whence, we write, for every $\vec{n} = (n_1, n_2) \in \mathbb{F}_p^2$, the function
\begin{equation}
 r_s(\vec{n}) = \sum_{a_1, \dots, a_s \in A} \mathds{1}_{a_1 + \dots + a_s  = n_1} \mathds{1}_{a_1^2 + \dots + a_s^2 = n_2}. 
\end{equation}
By double counting, we have that
\begin{equation} \label{dc12}
 \sum_{\vec{n} \in \mathbb{F}_p^2} r_s(\vec{n}) = |A|^s \ \ \text{and} \ \  \sum_{\vec{n} \in \mathbb{F}_p^2} r_s(\vec{n})^2 = J_s(A). 
 \end{equation}
Writing $\mathscr{A} = \{(a,a^2) : a \in A\}$ and $s\mathscr{A} = \{ \vec{a}_1 + \dots + \vec{a}_s : \vec{a}_1, \dots, \vec{a}_s \in \mathscr{A}\}$, we see that the function $r_s$ is supported on $s \mathscr{A}$. Setting $A = [N] (\mod p)$ for some $N < p$, we observe that $|s \mathscr{A}| \leq \min\{sNp, s^2N^3\}$. Thus, applying the Cauchy-Schwarz inequality, we get that 
\[ J_s(A) \geq \frac{|A|^{2s} }{|s \mathscr{A}|} \gg_s \frac{|A|^{2s-1} }{p} + |A|^{2s-3}. \]

Furthermore, noting the diagonal solutions $x_i = x_{i+s}$, for every $1 \leq i \leq s$, of \eqref{vino4}, we see that $J_s(A) \geq |A|^s$, and this finishes the proof of inequality \eqref{lbdl}. We conclude this section by presenting the proof of Proposition \ref{lset}.
 
\begin{proof}[Proof of Proposition \ref{lset}]
    Let $p \geq 10$ be a prime number, let  $P, X \subseteq \mathbb{F}_p^2$ be non-empty sets, let 
    \[ l_{\vec{x}} = \{ (t, (t - x_1)^2 + x_2) : t \in \mathbb{F}_p \} \]
    for every $\vec{x} = (x_1, x_2) \in\mathbb{F}_p^2$ and let $L_X = \{ l_{\vec{x}} : \vec{x} \in X\}$. Our first aim is to show that
    \begin{equation} \label{rdj}
 \sum_{\vec{x} \in X} \sum_{\vec{p}\in P} \mathds{1}_{\vec{p} \in l_{\vec{x}}} \ll |P||X|/p  \ + \ p^{1/2} |P|^{1/2}|X|^{1/2} .     \end{equation}
In order to prove this, we closely follow the ideas in \cite{MP2016}, and so, note that
\begin{align*}
\sum_{\vec{x} \in \mathbb{F}_p^2} \Big(  \sum_{\vec{p} \in P}    \mathds{1}_{\vec{p} \in l_{\vec{x}}} \Big)^2
 &  =   \sum_{\vec{p} \in P} \sum_{\vec{x} \in \mathbb{F}_p^2}  \mathds{1}_{\vec{p} \in l_{\vec{x}}}   +    \sum_{\vec{p} \neq \vec{p}' \in P}   \sum_{\vec{x} \in \mathbb{F}_p^2}\mathds{1}_{\vec{p} \in l_{\vec{x}}} \mathds{1}_{\vec{p}' \in l_{\vec{x}}}  \\
 &  = p |P| + |P|(|P|-1),
\end{align*}
where the last step follows from the facts that given two distinct points $\vec{p}, \vec{p}' \in P$, there is exactly one $\vec{x} \in \mathbb{F}_p^2$ such that $\vec{p},\vec{p}' \in l_{\vec{x}}$ and that given $\vec{p} \in P$, there are exactly $p$ choices of $\vec{x} \in \mathbb{F}_p^2$ such that $\vec{p} \in l_{\vec{x}}$. This means that
\begin{align*}
\sum_{\vec{x} \in \mathbb{F}_p^2} \Big(  \sum_{\vec{p} \in P}    \mathds{1}_{\vec{p} \in l_{\vec{x}}}  - |P|/p\Big)^2  
& = \sum_{\vec{x} \in \mathbb{F}_p^2} \Big(  \sum_{\vec{p} \in P}    \mathds{1}_{\vec{p} \in l_{\vec{x}}} \Big)^2 - 2p^{-1}|P|\sum_{\vec{x} \in \mathbb{F}_p^2}  \sum_{\vec{p} \in P}    \mathds{1}_{\vec{p} \in l_{\vec{x}}}  + |P|^2   \\
& = p|P|  + |P|(|P|-1) - |P|^2 = (p-1)|P|.
\end{align*}
Applying the Cauchy-Schwarz inequality along with the above estimate, we find that
\begin{align*}
     \sum_{\vec{x} \in X} \sum_{\vec{p}\in P} \mathds{1}_{\vec{p} \in l_{\vec{x}}} - |P||X|/p 
     & = \sum_{\vec{x} \in X} \Big(  \sum_{\vec{p} \in P}    \mathds{1}_{\vec{p} \in l_{\vec{x}}}  - |P|/p\Big)  \\ 
     & \leq  |X|^{1/2} (p-1)^{1/2} |P|^{1/2},
 \end{align*}
 whence, the claimed estimate \eqref{rdj} follows.

We will now employ a weighted version of \eqref{rdj} to prove that for every $s \geq 3$ and $A \subseteq \mathbb{F}_p$, we have
\begin{equation} \label{udp}
J_s(A) \ll |A|^{2s-1}/p +  p J_{s-1}(A) (\log |A|)^2 . 
\end{equation} 
We point out that \eqref{udp} is only non-trivial when $p \leq |A|^2(\log |A|)^{-2}$ since we have the trivial bound $J_s(A) \leq |A|^2 J_{s-1}(A)$. Moreover, \eqref{udp} may be inductively applied along with the trivial bound $J_{2}(A) \ll |A|^2$ to deduce the desired estimate
\[ J_s(A)  \ll |A|^{2s-1}/p + p^{s-2}|A|^2 (\log |A|)^{2s}  \]
whenever $s \geq 3$, whence our main aim now is to prove \eqref{udp}. Writing
\[ r_j(\vec{n}) = \sum_{\vec{a}_1, \dots, \vec{a}_j \in \mathscr{A}} \mathds{1}_{\vec{n} = \vec{a}_1 + \dots + \vec{a}_j}  \]
for every $j \in \mathbb{N}$ and $\vec{n} \in \mathbb{F}_p^2$, we observe that
\begin{align*}
J_s(A) 
& = \sum_{\vec{a}_1, \dots, \vec{a}_{2s}\in \mathscr{A}}   \mathds{1}_{\vec{a}_1 + \dots + \vec{a}_s - \vec{a}_{s+1} - \dots - \vec{a}_{2s-1} = \vec{a}_{2s}} 
 = \sum_{\vec{p} \in s \mathscr{A} } \sum_{ \vec{x} \in (s-1)\mathscr{A}} r_s(\vec{p}) r_{s-1}(\vec{x}) \mathds{1}_{\vec{p} - \vec{x} \in \mathscr{A} }  \\ 
& \leq  \sum_{\vec{p} \in s \mathscr{A} } \sum_{ \vec{x} \in (s-1)\mathscr{A}} r_s(\vec{p}) r_{s-1}(\vec{x}) \mathds{1}_{\vec{p} \in l_{\vec{x}} } .
\end{align*}  
Let $J, K \in \mathbb{N}$ satisfy $2^{J-1} \leq |A|^{s}  < 2^J$ and $2^{K-1} \leq |A|^{s-1}  < 2^K$. Moreover, given any $j,k \in \mathbb{N}$, we denote 
\[ P_j = \{  \vec{p} \in s\mathscr{A} :  2^{j-1} \leq r_{s}(\vec{p}) < 2^j\}   \ \text{and} \ X_k = \{  \vec{x} \in (s-1)\mathscr{A} :  2^{k-1} \leq r_{s}(\vec{x}) < 2^k  \} . \]
With this in hand, we see that
\begin{align}
 J_s(A) 
 & \leq \sum_{j=0}^J \sum_{k=0}^K  \sum_{\vec{p} \in P_j } \sum_{ \vec{x} \in X_j} r_s(\vec{p}) r_{s-1}(\vec{x}) \mathds{1}_{\vec{p} \in l_{\vec{x}} }  
 \ll \sum_{j=0}^J \sum_{k=0}^K 2^j 2^k \sum_{\vec{p} \in P_j } \sum_{ \vec{x} \in X_j} \mathds{1}_{\vec{p} \in l_{\vec{x}} } \nonumber \\
 & \ll  \sum_{j=0}^J \sum_{k=0}^K 2^j 2^k  \big( |P_j||X_k|/p + p^{1/2} |P_j|^{1/2} |X_j|^{1/2}  \big)
\nonumber \\
 & \ll p^{-1}  \sum_{j=0}^J 2^j |P_j|  \sum_{k=0}^K 2^k |X_k|   +  p^{1/2} \sum_{j=0}^J 2^j |P_j|^{1/2}  \sum_{k=0}^K 2^k |X_k|^{1/2} \label{huf}
\end{align}
where the third inequality follows from \eqref{rdj}. Noting \eqref{dc12}, we may deduce that
\[ \sum_{j=0}^J 2^j |P_j| \ll  |A|^s \ \ \text{and} \ \  \sum_{j=0}^J 2^j |P_j|^{1/2} \leq J^{1/2} (\sum_{j=0}^J 2^{2j} |P_j|)^{1/2} \ll_s (\log |A|)^{1/2} J_s(A)^{1/2} .\]
      Substituting these and the corresponding upper bounds for $\sum_{k=1}^K 2^k|X_k|$ and  $\sum_{k=1}^K 2^k|X_k|^{1/2}$ in \eqref{huf} dispenses the estimate
      \[ J_s(A) \ll_s p^{-1} |A|^{2s-1} + p^{1/2} (\log |A|) J_s(A)^{1/2} J_{s-1}(A)^{1/2} ,\]
      which, in turn, simplifies to give \eqref{udp}. Thus, we conclude the proof of Proposition \ref{lset}.
\end{proof}




\section{Proof of Theorem \ref{3c}}

As before, we consider a prime $p \geq 100$, since for all small primes $p$, our set $A \subseteq \mathbb{F}_p$ has $|A| \leq 100$, whence, we may use the trivial bound $J_s(A) \leq |A|^{2s-2}$ and if necessary, modify the implicit constant in the Vinogradov notation suitably to prove Theorem \ref{3c}. Moreover, recalling \eqref{ep3} and applying orthogonality, we see that
\begin{equation}\label{jj3}
 J_{s}(A) = \n{F_{\mathds{1}_A}}_{L^{2s}}^{2s}  \leq |A|^{2s - 6}  \n{F_{\mathds{1}_A}}_{L^6}^{6}  = |A|^{2s - 6} J_{3}(A),   
\end{equation}
whence it suffices to show that $J_{3}(A) \ll |A|^{4 - 1/9}.$

We now define, for every $m,n \in \mathbb{F}_p$, the quantity $r(m,n)$ to be the number of solutions to the system 
\begin{equation} \label{r1}
    a_1 + a_2 - a_3 = m \ \text{and} \ a_1^2 + a_2^2 - a_3^2 = n,
\end{equation}   
with $a_1, a_2, a_3 \in A$. Moreover, we consider the set
\[ \mathscr{A} = \{(a,a^2) : a \in A \}. \]
We see that $r(m,n)$ naturally relates to $J_3(A)$ via a standard double counting argument which gives us
\[ J_3(A) = \sum_{(m,n) \in 2 \mathscr{A} - \mathscr{A} } r(m,n)^2.  \]

We perform some preliminary pruning, whence, let 
\[ S_1  = \{ (m,n) \in 2 \mathscr{A} - \mathscr{A} : m = 0\} \ \ \text{and} \  S_2  =  \{ (m,n) \in 2 \mathscr{A} - \mathscr{A} : n = 0\} \]
and
\[ S_3  =  \{ (m,n) \in 2 \mathscr{A} - \mathscr{A} :  m^2 - n = 0\}. \]
We claim that
\begin{equation} \label{trivial}
    \sum_{(m,n) \in S_i} r(m,n)^2 \ll |A|^3,
\end{equation}  
for each $1 \leq i \leq 3$. In order to see the $i=1$ case, note that the left hand side in \eqref{trivial} for $i=1$ is bounded above by the number of solutions to the system of equations
\[ a_1 + a_2 - a_3 = a_4 + a_5 - a_6 = 0 \ \text{and} \   a_1^2 + a_2^2 - a_3^2 = a_4^2 + a_5^2 - a_6^2 , \]
with $a_1, \dots, a_6 \in A$. Furthermore, we observe that fixing $a_3, a_5, a_6 \in A$ fixes the value of $a_4$, which in turn, gives $O(1)$ admissible values for $a_1, a_2$, whenceforth, our claim is shown to be true. Similarly, one can prove the $i=2$ case of \eqref{trivial}. We now consider the $i=3$ case, where we observe that the left hand side of $\eqref{trivial}$ when $i=3$ is bounded above by the number of solutions to the system
\[ (a_1 + a_2 - a_3)^2 = a_1^2 + a_2^2 - a_3^2 = a_4^2 + a_5^2 - a_6^2 = (a_4 + a_5 - a_6)^2 , \]
with $a_1, \dots, a_6 \in A$. As before, fixing $a_4, a_5 \in A$ delivers a quadratic equation in $a_6$, whence, there are at most $O(1)$ admissible values of $a_6$. Fixing further the value of $a_3 \in A$ would then deliver at most $O(1)$ admissible values of $a_1, a_2$, and so, we are done.

Thus, writing 
\[ S = \{(m,n) \in 2 \mathscr{A} - \mathscr{A} : m \neq 0 \ \text{and} \ n \neq 0 \ \text{and} \ m^2 - n \neq 0 \}, \]
the preceding discussion implies that
\[ J_{3}(A) \ll |A|^3 + \sum_{(m,n) \in S} r(m,n)^2 .   \]
We may split $S = U \cup V$, where $\Delta = |A|^{8/9}$ and
\[ U = \{ (m,n) \in S : r(m,n) \leq \Delta\}   \ \text{and} \ V = S \setminus U.  \]
Note that
\[ \sum_{(m,n) \in U} r(m,n)^2  \leq \Delta \sum_{(m,n) \in U} r(m,n) \leq \Delta |A|^3 \ll |A|^{3 + 8/9}, \]
and so, it suffices to prove a similar estimate for $\sum_{(m,n) \in V} r(m,n)^2$.
We now analyse $r(m,n)$ in some more detail, and note that whenever $a_1, a_2, a_3 \in A$ satisfy $\eqref{r1}$, then we have that
\[   2(a_1 - m)(a_2- m) = m^2-n . \]
In particular the point $(a_1, a_2) \in A^2$ lies on the hyperbola $h_{m,n}$ given by
\[ h_{m,n} = \{(x,y) \in \mathbb{F}_p \times \mathbb{F}_p :  (m- x)(m-y) = (m^2 - n)/2 \}.  \]
Here, we see that there is a natural bijection from pairs $(m,n) \in V$ to the set of hyperbolae $H = \{ h_{m,n} : (m,n) \in V\}$. Thus, we have that
\begin{equation} \label{ref2}
      \sum_{(m,n) \in V} r(m,n)^2 \leq \sum_{(m,n) \in V}  |h_{m,n} \cap (A \times A)|^2   . 
\end{equation} 
We now partition $V$ as $V = V_1 \cup \dots \cup V_r$, for some $r \in \mathbb{N}$ satisfying 
\begin{equation} \label{r3}
    2^r \Delta \ll |A|,
\end{equation} 
such that for every $1 \leq i \leq r$, we have
\[ V_i = \{ (m,n) \in V : 2^{i-1} \Delta < |h_{m,n} \cap (A \times A)| \leq 2^i \Delta \}. \]
Our aim now will be to show that for every $1 \leq i \leq r$, we have
\begin{equation} \label{ref1}
    |V_i| \ll |A|^7 (2^i \Delta)^{-11/2}. 
\end{equation} 
Note that \eqref{ref1} implies that
\begin{align*}
\sum_{(m,n) \in V}  |h_{m,n} \cap (A \times A)|^2   
& = \sum_{1 \leq i \leq r} \sum_{(m,n) \in V_i}  |h_{m,n} \cap (A \times A)|^2  \\
& \ll \sum_{1 \leq i \leq r} |V_i| (2^i \Delta)^2 \ll \sum_{1 \leq i \leq r} |A|^7 (2^i \Delta)^{-7/2} \\
& \ll |A|^7 \Delta^{-7/2}  = |A|^{4 - 1/9},
\end{align*} 
which, when combined with \eqref{ref2}, delivers the desired bound.

Thus, we will now focus on proving that \eqref{ref1} holds true, and we will show this by the means of the following incidence estimate.

\begin{lemma} \label{hypinc}
  Let $W$ be a finite, non-empty subset of $V$ such that $|W| \gg |A|^{3/2}$. Then we have
  \[ \sum_{(m,n) \in W} |h_{m,n} \cap (A\times A)| \ll |A|^{14/11}|W|^{9/11}. \]
\end{lemma}

Observe that Lemma \ref{hypinc} delivers \eqref{ref1} in a straightforward manner. In order to see this, we first confirm that if $|V_i| \ll |A|^{3/2}$, then
\[ |V_i| \ll |A|^7 |A|^{-11/2} \ll |A|^7 (2^i \Delta)^{-11/2}   \]
holds true by considering \eqref{r3}. On the other hand, if $|V_i| \gg |A|^{3/2}$, then we can apply Lemma \ref{hypinc} to deduce that
\[ |V_i| 2^i \Delta \leq \sum_{(m,n) \in V_i} |h_{m,n} \cap (A\times A)| \ll  |A|^{14/11}|V_i|^{9/11} , \]
which, in turn, simplifies to dispense \eqref{ref1}.

Thus, our next aim is to present the proof of Lemma \ref{hypinc}, and we pursue this in the next section.


\section{Incidence estimates for hyperbolae}

As previously mentioned, our main aim for this section will be to present the proof of Lemma \ref{hypinc}, parts of which will follow the circle of ideas explored in \cite{RW2022, Shk2021}. In this endeavour, we will be required to prove a variety of incidence estimates, and so, we present some further definitions. Given $(m,n) \in V$, we define the M\"{o}bius transformation map $g_{m,n} : \mathbb{F}_p \to \mathbb{F}_p$ as
\[ g_{m,n}(x) =  \frac{ m x  - (m^2 +n)/2}{x - m}  . \]
This is well-defined since for any $(m,n) \in V$, we have $m \neq 0$ and $n - m^2 \neq 0$. Moreover, as discussed in \S2, this M\"{o}bius transformation map can be seen to be affiliated with the matrix 
\[ M_{m,n} =  \begin{pmatrix}
m & - (m^2 +n)/2 \\
1 & -m  
\end{pmatrix}. \]
Note that if $(m,n) \in V$ and $(a_1, a_2) \in h_{m,n}$, then $a_2 = g_{m,n}(a_1).$

We also require some more notation, whence, let $\mathscr{L}_1$ be the set of all the vertical lines in $\mathbb{F}_p \times \mathbb{F}_p$, that is, lines $l$ of the form $x=m$ for some $m \in \mathbb{F}_p$. Similarly let $\mathscr{L}_2$ be the set of all parabolae in $\mathbb{F}_p \times \mathbb{F}_p$ given by the equation $y = -(x-\alpha)^2 + \beta$, for some $\alpha, \beta \in \mathbb{F}_p$.

\begin{lemma} \label{hypen}
Let $\tau >0$ be a real number and let $V'$ be a non-empty subset of $V$ such that 
\[  \max_{l \in \mathscr{L}_1 \cup \mathscr{L}_2} |l \cap V'| \leq \tau. \]
Then, writing $H = \{ g_{m,n} : (m,n) \in V'\}$, we have that
\[ E(H) := | \{ h_1^{-1} \circ h_2 = h_3^{-1} \circ h_4 : h_1, \dots, h_4 \in H\}| \ll |H|^2 \tau.  \]
\end{lemma}

With the above lemma in hand, we are now ready to present the proof of Lemma \ref{hypinc}.

\begin{proof}[Proof of Lemma \ref{hypinc}]

We will first define two parameters $\tau, \eta >0$ as
\[ \tau = |A|^{8/11}|W|^{2/11}  \ \ \text{and} \ \ \eta = \tau^{1/5} |A|^{7/5} |W|^{-2/5}  = |A|^{17/11} |W|^{-4/11}.  \]
Since $|W| \leq |V| \leq |A|^3$, we see that $\tau, \eta \gg 1$. Our first step involves an iterative pruning of our set $W$ which removes large intersections with translates of the parabola $y = -x^2$. In particular, set $W_0 = W$ and check whether there exists some $l \in \mathscr{L}_2$ such that $|l \cap W_0| > \tau$. If there exists such an $l$, let $W_1 = W \setminus W_0$ and repeat the preceding step for $W_1$. If not, we end the algorithm. We now suppose that this algorithm stops in $r$ steps, that is, we partition $W$ as
\[ W = W' \cup C_1 \cup \dots \cup C_r, \]
such that 
\[ \max_{l \in \mathscr{L}_2} |l \cap W'| \leq \tau \]
and such that for each $1 \leq i \leq r$, we have $C_i \subseteq l_i$ for some unique $l_i \in \mathscr{L}_2$ as well as that $|C_i| > \tau$. The latter condition implies that
\[ r < |W| / \tau. \]
Suppose that for each $1 \leq i \leq r$, the curve $l_i$ is denoted by the equation
\[ y = -(x - \alpha_i)^2 + \beta_i ,\]
for some $\alpha_1,\dots, \alpha_r,\beta_1,\dots,\beta_r \in \mathbb{F}_p$. Note that
\[ \sum_{(m,n) \in C_i} |h_{m,n }\cap (A \times A) |   \]
is bounded above by the number of solutions to the system of equations
\[   (m- a_1)(m-a_2) =  (m^2 - n)/2 \ \ \text{and} \ \ n = -(m - \alpha_i)^2 + \beta_i,  \]
with $a_1, a_2 \in A$ and $(m,n) \in C_i$. The above system implies that
\begin{equation} \label{pur}
     2m(\alpha_i - a_1 - a_2) + (2a_1a_2 - \alpha_i^2 + \beta_i) = 0.
\end{equation} 
Thus, for every $1 \leq i \leq r$, let $\mathcal{A}_i$ denote the set of all $(a_1, a_2) \in A\times A$ such that 
\[ a_1 + a_2 = \alpha_i \ \ \text{and} \ \ 2a_1a_2 = \alpha_i^2 - \beta_i . \]
Note that $|\mathcal{A}_i| \ll 1$ for every $1 \leq i \leq r$. Moreover, for every fixed $(a_1, a_2) \in (A \times A) \setminus \mathcal{A}_i$, there are at most $O(1)$ solutions to \eqref{pur}. Thus for each $1 \leq i \leq r$, we have
\begin{align*}
   \sum_{(m,n) \in C_i} |h_{m,n }\cap (A \times A) |   \leq  |A|^2 + |\mathcal{A}_i| |\mathcal{C}_i|.
\end{align*}
Summing this over all $1 \leq i \leq r$, we get
\begin{align*}
    \sum_{1 \leq i \leq r}  \sum_{(m,n) \in C_i} |h_{m,n }\cap (A \times A) | 
    &  \ll |A|^2 r + \sum_{1 \leq i \leq r} |\mathcal{C}_i |  \ll |A|^2 |W| / \tau + |W| \\ 
    &  \ll |A|^{14/11}|W|^{9/11}  ,
\end{align*}
with the last bound following from the fact that $|W| \leq |V| \leq |2 \mathscr{A} - \mathscr{A}| \leq |A|^3$. This is the desired upper bound so we now focus on the set $W'$.

 We perform some further pruning, and so, let $l_1, \dots, l_{s}$ be lines in $\mathscr{L}_1$ such that
\[ |l_i \cap W'| > \tau  \]
for every $1 \leq i \leq s$. As before, we have that $s < |W'|/ \tau$. Moreover, we may proceed as in the preceding step to deduce that
\begin{align*}
\sum_{1 \leq i \leq s} \sum_{(m,n) \in l_i \cap W'}   |h_{m,n} \cap (A \times A)| \leq s |A|^2 \ll |W||A|^2/ \tau \ll |A|^{14/11} |W|^{19/11}.  
\end{align*}
This is the required upper bound, and so, it suffices to focus on the set $W'' = W' \setminus (l_1 \cup \dots \cup l_s)$.

We commence our analysis for $W''$ by applying the Cauchy-Schwarz inequality to get
\begin{align*} 
\sum_{(m,n) \in W''}  \sum_{a_1, a_2 \in A} \mathds{1}_{(a_1, a_2) \in h_{m,n}  } 
& =   \sum_{ a_2 \in A} \sum_{(m,n) \in W''} 
 \sum_{a_1 \in A} \mathds{1}_{a_2 = g_{m,n}(a_1)  }  \\
 & \leq |A|^{1/2} \Big( \sum_{ a_2 \in A} \Big( \sum_{(m,n) \in W''} 
 \sum_{a_1 \in A} \mathds{1}_{a_2 = g_{m,n}(a_1)  } \Big)^2 \Big)^{1/2} \\
 &  = |A|^{1/2} \Big(  \sum_{(m,n), (m',n') \in W''}  \sum_{a_1,a_1' \in A} \sum_{a_2 \in A} \mathds{1}_{a_2 = g_{m,n}(a_1) = g_{m',n'}(a_1') }  \Big)^{1/2} \\
 & \leq |A|^{1/2} \Big(  \sum_{(m,n), (m',n') \in W''}  \sum_{a_1,a_1' \in A} \mathds{1}_{g_{m,n}(a_1) = g_{m',n'}(a_1')}  \Big)^{1/2}.
 \end{align*}
Thus, denoting
\[ \mathcal{M} =   \sum_{(m,n), (m',n') \in W''}  \sum_{a_1,a_1' \in A} \mathds{1}_{g_{m,n}(a_1) = g_{m',n'}(a_1')}  ,\]
we claim that
\begin{equation} \label{jk}
    \mathcal{M} \ll \tau^{1/5} |A|^{7/5}|W|^{8/5} .
\end{equation}
Inserting the claimed upper bound in the preceding discussion immediately gives us the desired estimate
\[ \sum_{(m,n) \in W''}  \sum_{a_1, a_2 \in A} \mathds{1}_{(a_1, a_2) \in h_{m,n}  }  \leq |A|^{1/2}\mathcal{M}^{1/2}  \ll |A|^{6/5 }|W|^{4/5} \tau^{1/10}  = |A|^{14/11} |W|^{9/11}, \]
and so, we now focus on proving \eqref{jk}.

We begin the above endeavour by defining
\[ H = \{ g_{m,n} : (m,n) \in W''\} \]
and 
\[ q(g) = |\{ (h_1, h_2) \in H \times H : g = h_1^{-1} \circ h_2 \}| \]
for any M\"{o}bius transformation $g$, which in turn gives us
\begin{align*}
     \mathcal{M} =  \sum_{g \in H \cdot H^{-1}} q(g) \sum_{a,a' \in A} \mathds{1}_{a = g(a')}.
\end{align*} 
We now partition the set $H \cdot H^{-1}$ as $ H\cdot H^{-1} = H' \cup H''$, where
\[ H' = \{ g \in H \cdot H^{-1} : \sum_{a,a' \in A} \mathds{1}_{a = g(a')} > \eta \} \ \ \text{and} \ \  H'' = (H \cdot H^{-1}) \setminus H_1 . \]
We see that
\[  \sum_{g \in H''} q(g) \sum_{a,a' \in A} \mathds{1}_{a = g(a')} \leq \eta \sum_{g \in H''} q(g) \leq  \eta |W|^2 = \tau^{1/5} |A|^{7/5}|W|^{8/5}, \]
which matches the upper bound in \eqref{jk}. Thus, it suffices to bound the contribution to $\mathcal{M}$ from $H'$.

Applying the Cauchy-Schwarz inequality, we infer that
\begin{equation} \label{fhr}
     \sum_{g \in H'} q(g) \sum_{a,a' \in A} \mathds{1}_{a = g(a')} 
    \leq \Big( \sum_{g \in H'} q(g)^2 \Big)^{1/2} \Big( \sum_{g \in H'}   \Big(\sum_{a,a' \in A} \mathds{1}_{a = g(a')} \Big)^2 \Big)^{1/2} .
\end{equation}   
Since
\[ \max_{l \in \mathscr{L}_1 \cup \mathscr{L}_2} |W'' \cap l| \leq \tau,   \]
we may bound the first factor on the right hand side in \eqref{fhr} by applying Lemma \ref{hypen}. In particular, this gives us
\begin{equation} \label{rus}
     \sum_{g \in H'} q(g)^2  \leq E(H) \ll  \tau |H|^2 
= |A|^{8/11} |W|^{24/11} .
\end{equation}  
The second factor on the right hand side in \eqref{fhr} may be estimated by noting that
\[ \sum_{g \in H'}   \Big(\sum_{a,a' \in A} \mathds{1}_{a = g(a')} \Big)^2   = \sum_{1 \leq i \leq r'} \sum_{g \in H_i} \Big(\sum_{a,a' \in A} \mathds{1}_{a = g(a')} \Big)^2 ,  \]
where $r'$ is some natural number satisfying
\[ \eta 2^{r'} \ll |A|,  \ \text{and}  \ H_i = \{ g \in H' : 2^{i-1} \eta < \sum_{a,a' \in A} \mathds{1}_{a = g(a')}  \leq 2^{i} \eta \} \]
for every $1 \leq i \leq r'$. We may now apply Lemma \ref{rich} to deduce that
\[ |H_i| \ll |A|^7 (2^i \eta)^{-5} \ \ \text{for every} \ 1 \leq i \leq r', \]
whereupon, we have
\[  \sum_{1 \leq i \leq r'} \sum_{g \in H_i} \Big(\sum_{a,a' \in A} \mathds{1}_{a = g(a')} \Big)^2 \ll  \sum_{1 \leq i \leq r'} |A|^7 (2^i \eta)^{-3} \ll |A|^7 \eta^{-3} . \]
Combining this with \eqref{fhr} and \eqref{rus}, we find that
\[  \sum_{g \in H'} q(g) \sum_{a,a' \in A} \mathds{1}_{a = g(a')}  \ll (\tau |H|^2 )^{1/2}(|A|^7 \eta^{-3})^{1/2}  
= \tau^{1/5} |W|^{8/5} |A|^{7/5}.
\]
which matches the required bound in \eqref{jk}. 
%
%
%
\end{proof}

Therefore, all that remains now is to prove Lemma \ref{hypen}, which is what we proceed with below.

\begin{proof}[Proof of Lemma \ref{hypen}]
    We start by noting that despite the lack of commutativity, one still has 
    \[ E(H) = | \{ h_1 \circ h_2^{-1} = h_3 \circ h_4^{-1} : h_1, \dots, h_4 \in H\}| .\]
Fixing $h_3, h_4 \in H$ in $|H|^2$ ways, our aim will be to show that the number of choices of $h_1, h_2 \in H$ with $h_1 \circ h_2^{-1} = h_3 \circ h_4^{-1}$ is at most $O(\tau)$. Recall that for any $(m,n) \in V$, we associate the M\"{o}bius transformation $g_{m,n}$ to the matrix
\[ M_{m,n} =  \begin{pmatrix}
m & - (m^2 +n)/2 \\
1 & -m  
\end{pmatrix}. \]
Let $h_1 = g_{m_1,n_1}$ and $h_2 = g_{m_2,n_2}$ for some $(m_1, n_1), (m_2,n_2) \in V'$. Denote $s_i = -(m_i^2 + n_i)/2$ for $i \in \{1,2\}$. Then, $h_2^{-1}$ is associated with the matrix 
\[ 
\begin{pmatrix}
m_2 & s_2  \\
1 & -m_2
\end{pmatrix}
\]
and consequently, $h_1 \circ h_2^{-1}$ is associated with the matrix
\[ \begin{pmatrix}
m_1 & s_1 \\
1 & -m_1  
\end{pmatrix} 
\cdot  \begin{pmatrix}
m_2 &  s_2 \\
1  & -m_2  
\end{pmatrix}^{-1} = 
 \begin{pmatrix}
m_1m_2 + s_1 &   m_1 s_2 - s_1 m_2\\
m_2 - m_1  &   s_2 + m_1m_2
\end{pmatrix} .\]
We divide our proof into two cases, the first being when $m_1 \neq m_2$. Now since $h_1 \circ h_2^{-1} = h_3 \circ h_4^{-1}$ and $h_3, h_4$ are fixed, we see that there exist fixed $\kappa_1, \kappa_2, \kappa_3 \in \mathbb{F}_p$ such that
\begin{align}
 m_1m_2 + s_1 & = (m_2 - m_1) \kappa_1,   \label{aeq1} \\
 m_1 s_2 - s_1 m_2 & = (m_2 - m_1) \kappa_2   , \label{aeq2} \\
  s_2 + m_1m_2  & = (m_2 - m_1) \kappa_3  .\label{aeq3}  
\end{align}
Substituting the values of $s_1$ and $s_2$ from \eqref{aeq1} and \eqref{aeq3} respectively into \eqref{aeq2}, we find that
\[ m_1m_2(m_2 - m_1)
= (m_2 - m_1) (m_2 \kappa_1 + \kappa_2 - m_1 \kappa_3). \]
Simplifying the above yields
\[  m_2 (m_1 - \kappa_1) = \kappa_2 - m_1 \kappa_3.\]
Note that we can not have $m_1 = \kappa_1$, since if so, then \eqref{aeq1} would give us
$s_1 = -m_1^2$, that is, $m_1^2 - n_1 = 0$, which would contradict the fact that $(m_1, n_1) \in V$. Now substituting the value of $m_2$ into \eqref{aeq1}, we get that
\[   s_1 = 
- \kappa_2 + m_1 ( \kappa_3- \kappa_1) . \]
Inserting the definition of $s_1$ in the above, we find that
\[ n_1 =  - m_1^2 + 2m_1 (\kappa_1 - \kappa_3)  + 2 \kappa_2 = -(m_1 - (\kappa_1 - \kappa_3))^2 + (\kappa_1 - \kappa_3)^2 + 2\kappa_2.  \]
Thus $(m_1,n_1)$ lie on a parabola of the form $y = - (x -\alpha)^2 +\beta$ for fixed values of $\alpha, \beta \in \mathbb{F}_p$, whereupon, we may apply the hypothesis of Lemma \ref{hypen} to deduce that there are at most $\tau$ admissible values of $m_1, n_1$. Moreover, fixing $m_1, n_1$ fixes $m_2, n_2$ by considering the above set of equations, whence, we have that the contribution of such solutions to $E(H)$ is $O(|H|^2 \tau)$.



We now analyse the case when $m_1 = m_2$. Here, we immediately see that $m_1 m_2 + s_1$ and $s_2 + m_1m_2$ must be non-zero since the matrix associated with $h_1\circ h_2^{-1}$ is invertible. Thus, as before, we may find fixed $\kappa_1, \kappa_2 \in \mathbb{F}_p \setminus \{0\}$ such that
\begin{align*}
  m_1^2 + s_1 = \kappa_1 m_1(s_2 - s_1) \ \ \text{and} \ \ m_1^2 + s_2 = \kappa_2 m_1 (s_2 - s_1).
\end{align*}
This implies that 
\[ 1 = (\kappa_2 - \kappa_1) m_1 ,\]
whence, we obtain the value of $m_1$, and consequently, $m_2$. We may now apply the hypothesis of Lemma \ref{hypen} to deduce that there are at most $\tau$ admissible values of $n_1$, which in turn, fixes the values of $n_2$ via the preceding set of equations since $\kappa_1,m_1 \neq 0$. Therefore, we have that the contribution of such solutions to $E(H)$ is also $O(|H|^2 \tau)$, and so, we conclude the proof of Lemma \ref{hypen}.
\end{proof}


\section{Proofs of Theorems \ref{4c} and \ref{smp}}

We begin this section by recording some notation, and so, given $\vec{x} \in \mathbb{F}_p^2$, we will denote
\[ l_{\vec{x}} =  \{(t,t^2) : t \in \mathbb{F}_p \} + \vec{x}. \]
Moreover, given a non-empty set $X \subseteq \mathbb{F}_p^2$, we define
\[ L_{X} =   \{l_{\vec{x}} : \vec{x} \in X\} .\]
Our starting point will be to use \eqref{sz2} to obtain the following incidence estimate for sets of points and translates of parabolae in $\mathbb{F}_p^2$.

\begin{lemma} \label{sztr}
Let $P, X$ be non-empty subsets of $\mathbb{F}_p^2$ such that $|X|^{13} \ll p^{15}$. Then, we have
\[ \sum_{\vec{p} \in P} \sum_{\vec{x} \in X} \mathds{1}_{\vec{p} \in l_{\vec{x}}} \ll |P|^{11/15} |X|^{11/15} + |P| + |X|.\]
\end{lemma}
\begin{proof}
Let $\phi : \mathbb{F}_p^2 \to \mathbb{F}_p^2$ satisfy $\phi(x,y) = (x, x^2 - y)$. This implies that $\phi(t+x_1, t^2+ x_2) = (t+x_1, 2x_1 t + x_1^2 - x_2)$ for every $x_1, x_2, t \in \mathbb{F}_p$, and so, we have $\phi(l_{\vec{x}}) = m_{\vec{x}}$, where $m_{\vec{x}}$ is the line defined as
\[ m_{\vec{x}} = \{ (t, 2x_1 t - x_1^2 - x_2) \ | \ t \in \mathbb{F}_p\}. \]
We further note that $\phi(\vec{x}) = \phi(\vec{y})$ if and only if $\vec{x} = \vec{y}$, for every $\vec{x}, \vec{y} \in \mathbb{F}_p^2$. Thus, writing $L_{X}' = \{ m_{\vec{x}} \ | \ \vec{x} \in X\}$, we see that
\[ \sum_{p \in P, l \in L_{\vec{x}}} \mathds{1}_{p \in l} =  \sum_{p' \in \phi(P), m \in L_{\vec{x}}'} \mathds{1}_{p' \in m}. \]
Noting that $|\phi(P)| = |P|$ and $|L_{X}'| = |L_{X}| = |X|$ and applying $\eqref{sz2}$, we obtain the desired conclusion. 
\end{proof}

We can further obtain the following weighted version of the above result by combining Lemmata \ref{sztr} and \ref{wt}.

\begin{lemma} \label{wtst}
Let $P, X$ be non-empty subsets of $\mathbb{F}_p^2$, let $w: P \to \mathbb{N}$ and $w' : X \to \mathbb{N}$ be functions. Moreover, let $|X|^{13} \ll p^{15}$. Then 
\[ \sum_{p \in P, \vec{x} \in X} \mathds{1}_{p \in l_{\vec{x}}} w(p) w'(\vec{x}) \ll \n{w}_{2}^{8/15} \n{w}_1^{7/15} \n{w'}_2^{8/15} \n{w'}_1^{7/15} + \n{w'}_{\infty} \n{w}_1 + \n{w'}_1 \n{w}_{\infty}. \]
\end{lemma}

We now proceed to proving estimates on $J_{s}(A)$. We commence by recall the notation $\mathscr{A} = \{(a,a^2) : a \in A\}$ and that for any $s \in \mathbb{N}$ and $\vec{n} \in \mathbb{F}_p^2$, we define
\[ r_{s}(\vec{n}) = \{ (\vec{a}_1, \dots, \vec{a}_{s}) \in \mathscr{A}^s \ | \ \vec{n} = \vec{a}_1 + \dots + \vec{a}_{s} \} |. \]
Note that for any $s \geq 2$, we have
\begin{equation} \label{pre1}   
\sup_{\vec{n}} r_s(\vec{n}) \leq |A|^{2s - 2} \sup_{\vec{n}} r_2(\vec{n})  \ll |A|^{2s-2}.  
\end{equation} 
With this in hand, we record the main set of iterative bounds for $J_s(A)$.

\begin{lemma} \label{inc1}
Let $s \geq 3$, let $A$ be a subset of $\mathbb{F}_p$ such that $|A| \ll_{s} p^{\frac{15}{13(s-1)}}$. Then
\[ J_{s}(A) \ll_{s} |A|^{\frac{14s-7}{11}} J_{s-1}(A)^{\frac{4}{11}} + |A|^{2s-3}. \]

\end{lemma}
\begin{proof}
We begin by noting that
\[ J_{s}(A) = \sum_{\vec{a}_1, \dots, \vec{a}_{2s} \in \mathscr{A}} \mathds{1}_{\vec{a}_1 + \dots + \vec{a}_{s} = \vec{a}_{s+1} + \dots + \vec{a}_{2s}} = \sum_{\vec{u} \in s\mathscr{A}} \sum_{\vec{v} \in (s-1)\mathscr{A}} \sum_{\vec{a} \in \mathscr{A}} r_{s}(\vec{u}) r_{s-1}(\vec{v})  \mathds{1}_{\vec{u} = \vec{v} + \vec{a}}. \]
Moreover, since $\vec{u} = \vec{v} + \vec{a}$ implies that $\vec{u} \in l_{\vec{v}}$, we see that
\[ J_{s}(A) \leq \sum_{\vec{u} \in s\mathscr{A}} \sum_{\vec{v} \in (s-1)\mathscr{A}} \sum_{\vec{a} \in \mathscr{A}} r_{s}(\vec{u}) r_{s-1}(\vec{v}) \mathds{1}_{\vec{u} \in l_{\vec{v}}}. \]
Noting the fact that $|(s-1)\mathscr{A}|^{13} \leq |A|^{13(s-1)} \ll_{s} p^{15}$, we may use Lemma $\ref{wtst}$ to bound the right hand side above. This gives us
\[ J_{s}(A) \ll \n{r_s}_2^{8/15} \n{r_s}_1^{7/15} \n{r_{s-1}}_2^{8/15} \n{r_{s-1}}_1^{7/15} + \n{r_{s}}_{\infty} \n{r_{s-1}}_1 + \n{r_{s}}_{1} \n{r_{s-1}}_{\infty}. \]
As in \S2, a standard double-counting argument gives us
\[ \n{r_j}_2^2 = \sum_{\vec{u} \in j \mathscr{A}} r_{j}(\vec{u})^2 = J_{j}(A) \ \text{and} \ \n{r_j}_1 = \sum_{\vec{u} \in j \mathscr{A}} r_{j}(\vec{u}) = |A|^j \]
for each $j \in \mathbb{N}$. Additionally, using $\eqref{pre1}$, we see that
\[ \n{r_j}_{\infty} = \sup_{\vec{n} \in j \mathscr{A}} r_{j}(\vec{n}) \leq |A|^{j-2}. \]
Thus, we deduce that
\[ J_{s}(A) \ll J_{s}(A)^{4/15} J_{s-1}(A)^{4/15} |A|^{\frac{14s-7}{15}} + |A|^{2s-3}, \]
which, upon simplifying, delivers the bound
\[ J_{s}(A) \ll J_{s-1}(A)^{\frac{4}{11}} |A|^{\frac{14s-7}{11}} + |A|^{2s-3}. \qedhere \]
\end{proof}

We may now iterate this with the fact that $J_{3}(A) \ll |A|^{4-1/9}$ to prove Theorem \ref{4c}.


\begin{proof}[Proof of Theorem \ref{4c}]
We prove this by induction on $s$, and so, our base case is when $s=4$. Thus, we have $A \subseteq \mathbb{F}_p$ such that $|A| \ll p^{15/39} \leq p^{1/2}$, whence, we may apply Theorem \ref{3c} to deduce that 
\[ J_3(A) \ll |A|^{4 - 1/9} .\]
Combining the estimates from Lemma \ref{inc1} along with the above upper bound gives us
\[ J_4(A) \ll J_3(A)^{4/11} |A|^{49/11} + |A|^5 = |A|^{6 - 13/99} = |A|^{6 - 1/7 + (4/11)\cdot (2/63)} ,\]
which is the claimed estimate. We now proceed with the inductive step, and so, we suppose that 
\[ J_s(A) \ll_s |A|^{2s -  2 - 1/7 + (4/11)^{s-3} \cdot (2/63)} \]
for any $A \subseteq \mathbb{F}_p$ with $|A| \ll p^{15/13(s-1)}$. Now given a set $A \subseteq \mathbb{F}_p$ with $|A| \ll p^{15/13s}$, we combine the preceding bound with Lemma \ref{inc1} to deduce that
\begin{align*}
  J_s(A) 
  & \ll_s |A|^{(8s-8)/11 - 4/77 + (4/11)^{s-2} \cdot (2/63)} |A|^{(14s+7)/11} + |A|^{2s - 1} \\
  & |A|^{2s - 1/7 + (4/11)^{s-2} \cdot (2/63 )} ,
\end{align*}
which is the required bound. Thus, we conclude the proof of Theorem \ref{4c}.
\end{proof}

We conclude this section by providing the proof of Theorem \ref{smp}.

\begin{proof}[Proof of Theorem \ref{smp}]
We begin by observing that it suffices to consider the case when $A \subseteq \mathbb{F}_p \setminus \{0\}$. In order to see this, suppose that $a_i = 0$ for some $1 \leq i \leq 6$. By considering the multiplicative equation $a_1 a_2 a_3 = a_4 a_5 a_6$, we immediately see that if one of $a_1, a_2, a_3$ is zero, then so is one of $a_4, a_5, a_6$, whence, up to losing a factor of $9$, it suffices to consider the case when $a_1 = a_4 = 0$. Inserting this into the additive equation in \eqref{fv1}, we find that $a_2 + a_3 = a_5 + a_6$, and so, there are at most $O(|A|^3)$ such solutions, which is much less than the desired bound.

Thus assuming $A$ to be a subset of $\mathbb{F}_p \setminus \{0\}$, we consider, for each $(m,n) \in \mathbb{F}_p^2$, the system of equations
\begin{equation} \label{nvm}
m = a_1 + a_2 + a_3 \ \ \text{and} \ \ n = a_1 a_2 a_3 , 
\end{equation} 
with $a_1, a_2, a_3 \in A$. This implies that
\[  (a_1 + a_2) + n(a_1 a_2)^{-1} = m. \]
We denote $s(m,n)$ to be the number of solutions to \eqref{nvm} with $a_1, a_2, a_3 \in A$ and we define the sets
\[ S_1 = \{ (a_1 + a_2, (a_1a_2)^{-1}) : a_1, a_2 \in A\}  \ \ \text{and} \ \  S_2 = \{(a_1 + a_2 + a_3, a_1a_2a_3) : a_1, a_2, a_3 \in A\}. \]
Since there are at most $O(1)$ choices of $(a_1, a_2) \in A^2$ which satisfy $a_1 + a_2$ being some fixed element of $\mathbb{F}_p$ and $(a_1a_2)^{-1}$ being some fixed element of $\mathbb{F}_p \setminus \{0\}$, we see that
\[
    T(A) 
     = \sum_{(m,n) \in S_2} s(m,n)^2 \ll \sum_{(m,n) \in S_2} s(m,n) \sum_{(u,v) \in S_1}  \mathds{1}_{(m,n) \in l_{u,v}  }  ,\]
where for each $u \in \mathbb{F}_p$ and $v \in \mathbb{F}_p\setminus \{0\}$, we define the line 
\[ l_{u,v} = \{ (x,y) \in \mathbb{F}_p^2 : x = y v + u \}.  \]
    Combining Lemma \ref{wt} along with \eqref{sz2} and the fact that $|S_1| \leq |A|^3 \ll p^{15/13}$, we may now deduce that
\[ T(A) \ll \n{s}_2^{8/15} \n{s}_1^{7/15} |S_1|^{11/15} + \n{s}_{\infty} |S_1| + \n{s}_{1}. \]  
Again, applying a double counting argument, we get that
\[ \sum_{(m,n) \in S_2} s(m,n)^2 = T(A) \ \ \text{and} \ \ \sum_{(m,n) \in S_2} s(m,n) = |A|^3 ,  \]
which may then be combined with the preceding discussion and the bounds 
\[ \max_{(m,n) \in S_2} s(m,n) \ll |A| \ \ \text{and} \ \ |S_1| \ll |A|^2 \]
to get that
\[ T(A) \ll T(A)^{4/15} |A|^{7/5}  |A|^{22/15} + |A|^3. \]
Simplifying the above delivers the desired bound
\[  T(A) \ll |A|^{4  - 1/11} . \qedhere \]

\end{proof}

\section{Proofs of Corollaries \ref{cor1} and \ref{cor2}}

We first present the proof of Corollary \ref{cor1}.

\begin{proof}[Proof of Corollary \ref{cor1}]
Given $s \in \mathbb{N}$, we denote 
\[ K_s(A) = \sum_{a_1, \dots, a_{2s} \in A} \mathds{1}_{a_1^2 + \dots - a_{2s}^2 = 0} . \]
Recalling \eqref{ep3} and applying orthogonality and the triangle inequality, we see that 
\begin{align*}
K_{s}(A) 
& = \sum_{n \in sA - sA} \sum_{a_1, \dots, a_{2s} \in A} \mathds{1}_{a_1^2 + \dots - a_{2s}^2 = 0} \mathds{1}_{a_1 + \dots - a_{2s} = n}  \\
& =  \sum_{n \in sA - sA}  p^{-2} \sum_{x,y \in \mathbb{F}_p} |F_{\mathds{1}_A}(x,y)|^{2s} e(-xn/p) \\
& \leq |sA - sA| p^{-2} \sum_{x,y \in \mathbb{F}_p} |F_{\mathds{1}_A}(x,y)|^{2s} \\
& = |sA - sA| J_{s}(A).
\end{align*}
Setting $s=3$ in the above and combining this with Theorem \ref{3c} and Lemma \ref{prin}, we see that
\[ K_3(A) \ll |3A - 3A||A|^{4 - 1/9} \leq K^6 |A|^{5 - 1/9} .\]
Writing $S = \{a^2 : a \in A\}$, we now apply the Cauchy-Schwarz inequality to deduce that
\[ |3S| \geq |S|^6 K_3(A)^{-1} \gg |A|^{1 + 1/9} K^{-6}. \]
This concludes the proof of Corollary \ref{cor1}.
\end{proof}

We now focus on the setting of Corollary \ref{cor2} and we proceed by presenting the proof of the upper bound in \eqref{vach}.
Thus, given function $\mathfrak{a} : \mathbb{F}_p \to \mathbb{R}$ supported on some set $A \subseteq \mathbb{F}_p$ such that $|A| \ll p^{1/2}$, we may apply the Cauchy-Schwarz inequality to deduce that for every $x, y \in \mathbb{F}_p$, we have
\[ |F_{\mathfrak{a}}(x,y)|^2 = |\sum_{n \in \mathbb{F}_p} \mathfrak{a}(n) e( (x n + y n^2)/p)|^2 \leq |A|  \n{\mathfrak{a}}_{2}^2 . \]
This implies that
\begin{align*}
    \n{F_{\mathfrak{a}}}_{L^6}^6  \leq \max_{x,y \in \mathbb{F}_p} |F_{\mathfrak{a}}(x,y)|^2 \n{F_{\mathfrak{a}}}_{L^4}^4 \leq |A|  \n{\mathfrak{a}}_{2}^2 \n{F_{\mathfrak{a}}}_{L^4}^4 .
\end{align*}  
Applying orthogonality, 
 we see that
 \[ \n{F_{\mathfrak{a}}}_{L^4}^4 = \sum_{a_1, \dots, a_4 \in A} \mathfrak{a}(a_1) \dots \mathfrak{a}(a_4)  \mathds{1}_{a_1 + a_2 - a_3 - a_4 = a_1^2 + a_2^2 - a_3^2 - a_4^2  = 0} \ll \big(\sum_{a \in A} |\mathfrak{a}(a)|^2 \big)^2 =  \n{\mathfrak{a}}_{2}^4, \]
which, in turn, combines with the preceding inequality to deliver the bound stated in \eqref{vach}.

We now present the proof of Corollary \ref{cor2}, and here, we closely follow the proof of \cite[Theorem 14]{CG2007}.

\begin{proof}[Proof of Corollary \ref{cor2}]
Observe that Theorem \ref{3c} implies that for any $A' \subseteq A$, we have
\begin{equation} \label{js1}
J_{3}(A')  \leq X |A'|^3, 
\end{equation}
where $X \ll |A|^{8/9}.$ Next, we note that the conclusion of Corollary \ref{cor2} is invariant under replacing the function $\mathfrak{a}$ by $\lambda \mathfrak{a}$, for any $\lambda >0$. In particular, this means that upon replacing $\mathfrak{a}$ by $\n{\mathfrak{a}}_{2}^{-1}\mathfrak{a}$, we may assume that $\n{\mathfrak{a}}_{2} = 1$ and that $|\mathfrak{a}(n)| \leq 1$ for all $n \in \mathbb{F}_p$. We now set 
\[ A_0 = \{ a \in A : |\mathfrak{a}(a)| \leq 1/|A|\}, \ \ \text{and} \ \ A_j = \{ a \in A : 2^{j-1}/|A| < |\mathfrak{a}(a)| \leq 2^{j}/|A| \}  \]
for every $1 \leq j \leq J$, for some positive integer $J \leq 2\ceil{\log |A|}$. Writing, for every $0 \leq j \leq J$, the function $\mathfrak{a}_j : \mathbb{F}_p \to \mathbb{C}$ as $\mathfrak{a}_j(x) = \mathfrak{a}(x)$ when $x \in A_j$ and $\mathfrak{a}_j(x) = 0$ for $x \in \mathbb{F}_p \setminus A_j$, we see that 
\begin{equation} \label{popc}
\mathfrak{a} = \sum_{j=0}^J \mathfrak{a}_j \ \ \text{and} \ \ F_{\mathfrak{a}}  = \sum_{j=0}^J F_{\mathfrak{a}_j}  \ \ \text{and} \ \ 1 = \n{\mathfrak{a}}_{2}^2 \gg \sum_{j=1}^J|A_j| 2^{2j} |A|^{-2}. 
\end{equation} 
Note that for every $1 \leq j \leq J$, we may apply orthogonality and \eqref{js1} to deduce that
\begin{align*}
\n{ F_{\mathfrak{a}_j }}_{L^6}^6
& = \sum_{a_1, \dots, a_6 \in A} \mathfrak{a}(a_1) \dots \mathfrak{a}(a_6) \mathds{1}_{a_1 + \dots - a_6 = a_1^2 + \dots - a_6^2 = 0}  \\
& \ll 2^{6j}|A|^{-6} J_{3}(A_j) \ll X 2^{6j} |A|^{-6} |A_j|^3 . 
\end{align*}
Moreover, we have the trivial bound $|F_{a_0}(x,y)| \leq |A|^{-1} |A_0| \leq 1$, whence, $\n{F_{\mathfrak{a}_0}}_{L^6} \leq 1.$ Amalgamating this with the preceding estimate, inequality \eqref{popc}, the triangle inequality and the Cauchy-Schwarz inequality, we see that
\begin{align*}
\n{F_{\mathfrak{a}}}_{L^6}
& \leq \sum_{j=0}^{J} \n{F_{\mathfrak{a}_j}}_{L^6} \leq 1 +  X^{1/6} \sum_{j=1}^{J} 2^j|A|^{-1} |A_j|^{1/2} \\
& \leq  1 + X^{1/6} J^{1/2} \big(\sum_{j=1}^J 2^{2j}|A|^{-2} |A_j|\big)^{1/2}  \\ 
& \ll X^{1/6}(\log|A| + 1)^{1/2} \ll |A|^{4/27} ( \log|A| + 1)^{1/2},
\end{align*} 
which is the desired bound. 
\end{proof}


\bibliographystyle{amsbracket}
\providecommand{\bysame}{\leavevmode\hbox to3em{\hrulefill}\thinspace}

\end{document}